\documentclass[11pt, letter]{article}
\usepackage{amsmath}
\usepackage{amsfonts}
\usepackage{amssymb}
\usepackage{amsthm}
\usepackage{graphicx}
\usepackage{hyperref}
\usepackage{caption, subcaption}
\usepackage{enumitem}
\newtheorem{Thm}{Theorem}[section]

\newtheorem{Lemma}[Thm]{Lemma}
\newtheorem{Prop}[Thm]{Proposition}

\theoremstyle{definition}

\theoremstyle{remark}
\newtheorem{Ex}[Thm]{Example}
\newtheorem{Rmk}[Thm]{Remark}
\numberwithin{equation}{section}

\begin{document}

\title{Spectral parameter power series representation for regular solutions of the radial Dirac system}
\author{E. Roque, S.M. Torba \\{\small Departamento de Matem\'{a}ticas, Cinvestav, Unidad Quer\'{e}taro, }\\{\small Libramiento Norponiente \#2000, Fracc. Real de Juriquilla,
Quer\'{e}taro, Qro., 76230 MEXICO.}\\{\small e-mail: earoque@math.cinvestav.mx, storba@math.cinvestav.edu.mx}}

\maketitle

\begin{abstract}
A spectral parameter power series (SPPS) representation for the regular solution of the radial Dirac system with complex coefficients is obtained, as well as a SPPS representation for the (entire) characteristic function of the corresponding spectral problem on a finite interval. Based on the SPPS representation, a numerical method for solving spectral problems is developed. It is shown that the method is also applicable to solving spectral problems for perturbed Bessel equations. We exhibit that the proposed numerical method delivers excellent results. Additionally, an application of the method to find the energy values of hydrogen-like atoms with a finite radius is presented.
\end{abstract}

\section{Introduction}

The radial Dirac equation for a time independent and spherically symmetric problem can be reduced to a pair of coupled equations \cite{rDirac} :
\begin{equation}\label{eqn:phy-dirac}
\begin{cases}
\frac{dF}{dx}&=-(\frac{\kappa}{x}+\frac{V_{ps}(x)}{\hbar} )F + \frac{mc^2+V_s(x)-V_v(x)}{\hbar c }G + \frac{E}{\hbar c}G \\
\frac{dG}{dx}&=  \frac{mc^2+V_s(x)+V_v(x)}{\hbar c }F  +(\frac{\kappa}{x}+\frac{V_{ps}(x)}{\hbar} )G - \frac{E}{\hbar c}F
\end{cases}
\end{equation}

The radial functions \(F \) and \(G \) are known respectively as the large and small components of the wave function \( \Psi(x) = (F(x), G(x))^T \). In the equations above, \( V_s \) is a scalar potential, \(V_v \) is the time-component of a vector potential and \(V_{ps} \) is a pseudoscalar potential. On the other hand, \( c \) stands for the velocity of light, \( \hbar \) is the reduced Planck constant and \( m \) is the mass of the particle. The parameter \(\kappa \) is the relativistic angular-momentum quantum number, and it is a nonzero integer. \\

The Dirac system arises in quantum mechanics and it governs 1/2-spin particles such as electrons. Besides all the well-known applications of the Dirac system in the theory of quantum mechanics, recently, even very particular cases of the system \eqref{eqn:phy-dirac} when both the scalar potential and the vector potential are constant, have found applications to bilayer graphene in the Jackiw-Pi model, see \cite{Jackiw}. Due to its importance, numerical methods for solving the radial Dirac equation have been widely studied, see \cite{Almana-hp,Almana-linear,SALVAT1995151, TSOULOS2019237, Silbar2011, mudirac} and the references therein.\\

For square-integrable potentials on a finite interval \( (0,a] \), the radial Dirac system possesses up to a mutiplicative factor exactly one solution that is bounded at the origin and it is a square-integrable function \cite{Hryniv}. Throught this paper, this solution is referred to as a regular solution. We give the precise definition in Subsection \ref{sec:regular}. \\

In the present work we obtain an analytical representation for the regular solution of the radial Dirac system \eqref{eqn:dirac} in the form of a spectral parameter power series (SPPS), whose coefficients can be computed by the means of a recursive integration procedure. The SPPS method has been widely used to solve analytically and numerically ordinary and partial differential equations, together with different types of boundary-value and spectral problems, see e.g. \cite{SPPS, Castillo}. \\

Recently, the SPPS method has been constructed for the one-dimensional Dirac equation by Guti\'errez and Torba in \cite{Nelson}, and extended by Barrera-Figueroa, Rabinovich and Loredo-Ram\'irez for potentials with Dirac deltas in \cite{Rabinovich}. It is worth mentioning that our work presents significant changes to that of Guti\'errez and Torba both in the methods of proof and the main formulas of what are known as formal powers. Particularly, the latter changes have a major impact on the resulting numerical accuracy of the present work. \\

The paper is organized as follows. In Section \ref{sec:spps}, we consider a radial Dirac system whose right-hand side may contain an arbitrary matrix-function coefficient at the spectral parameter. In Subsection \ref{sec:regular}, we introduce a recursive system of functions known as formal powers. The recursive procedure to construct the formal powers starts with a particular solution \( (u_p, v_p)^T \) corresponding to the zero value of the spectral parameter, such that \( u_p \) does not vanish except at \(x = 0\). At the end of this subsection, we show how the regular solution can be written in terms of these formal powers; see Theorem \ref{thm:spps}. \\

Afterward, in Subsection \ref{sec:shift} we describe the spectral shift procedure, which allows us to find the regular solution starting from a particular solution corresponding to the parameter \(\lambda = \lambda_0   \) for a given \(\lambda_0 \). This technique has been widely used to improve the numerical accuracy of the SPPS-based methods.\\

Later, in Subsection \ref{sec:particular}, we show how we can apply the SPPS method to construct a particular solution corresponding to the zero value of the spectral parameter for a continuous potential. Then, we extend the results in Subsection \ref{sec:singq} and show that our method can be used to compute a particular solution even for singular potentials satisfying a certain growth condition at the origin \(x = 0\). We end Section \ref{sec:spps} discussing the existence of particular solutions with a non-vanishing first entry in Subsection \ref{sec:nonvanishing}. \\

In Section \ref{sec:numerical}, we propose the numerical algorithm for applying the SPPS representation to the solution of spectral problems and present numerical results for two specific problems. First, we show how we can transform a spectral problem for the perturbed Bessel equation to a spectral problem for the radial Dirac system in Subsection \ref{sec:bessel}. Our approach yields excellent numerical results and outperforms the method presented in \cite{Castillo}. Finally, in Subsection \ref{sec:atom} we show how we can use the SPPS representation to compute the energy levels of an electron orbiting a hydrogen-like atom with a finite radius.

\section{The spectral parameter power series representation for the regular solution of the radial Dirac system}\label{sec:spps}

With the aim of generality, and looking to keep a consistent notation with the previously published work \cite{Nelson}, we consider the following system
\begin{equation}\label{eqn:dirac}
\begin{cases}
v'+p_1u+(\frac{\kappa}{x}+q )v=\lambda (r_{11} u + r_{12}v ), \qquad x\in (0,a] \\
-u'+(\frac{\kappa}{x}+q )u+p_2v=\lambda ( r_{21} u + r_{22}v ).
\end{cases}
\end{equation}
In matrix form, the system can be written as
\begin{equation}
B \frac{dY}{dx} + P(x)Y = \lambda R(x) Y, \qquad Y(x) = \begin{pmatrix} u(x) \\ v(x) \end{pmatrix}
\end{equation}
where
\[
B = \begin{pmatrix} 0 & 1 \\ -1 & 0 \end{pmatrix}, \quad
P(x)= \begin{pmatrix} p_1(x) & \frac{\kappa}{x} + q(x) \\
\frac{\kappa}{x} + q(x) & p_2(x)
\end{pmatrix}, \quad
R(x) = \begin{pmatrix}
r_{11}(x) & r_{12}(x) \\
r_{21}(x) & r_{22}(x)
\end{pmatrix}.
\]
It is assumed that the following conditions are satisfied:
\begin{enumerate}[label*=(C\arabic{*}), ref=(C\arabic{*})]
\item \( p_i, r_{ij}, q \in C( [0,a]) \), \(i,j \in \{ 1,2 \}  \) are complex-valued functions of the real variable \(x \). \label{cdt:prq}
\item \( \lambda  \) is an arbitrary complex constant. \label{cdt:lam}
\item  \(  p_1(0) \neq 0 \). \label{cdt:p10}
\item \( \kappa \geq \frac{1}{2} \). \label{cdt:kappa}
\end{enumerate}
\begin{Rmk}
Condition \ref{cdt:prq} can be weakened for the potential \( q \). It is sufficient to have \( q \in C( (0,a]) \) with the growth condition at the origin
\[
\left\vert q(x) \right\vert \leq c_q x^\alpha,
\]
for some \( c_q >0 \) and \( -1< \alpha \leq 0 \). For the ease of exposition, we assume \( q \in C([0,a]) \) throughout most of our work. The validity of the method for the case of a singular potential can be found in Section \ref{sec:singq}.
\end{Rmk}

\begin{Rmk}
 Although the assumption \(p_1(0) \neq 0 \) seems quite restrictive, it can be easily avoided as long as \( p_1 \) and \( r_{11} \) do not vanish simultaneously at \( x=0 \), i.e., \( \vert  p_1(0) \vert + \vert r_{11}(0) \vert \neq 0 \). For instance, if \( p_1(0)=0 \) and \( R \) is a symmetric matrix, adding \( R(u, v)^T \) to both sides of \eqref{eqn:dirac} we get
\[
\begin{cases}
v'+(p_1+r_{11})u+(\frac{\kappa}{x}+q +r_{12} )v=\Lambda (r_{11} u + r_{12}v ), \\
-u'+(\frac{\kappa}{x}+q+r_{21} )u+(p_2+r_{22})v=\Lambda ( r_{21} u + r_{22}v ),
\end{cases}
\]
and the condition is satisfied for a new spectral parameter \( \Lambda \) defined as \( \Lambda:=\lambda+1 \ \). For the general case, see Section \ref{sec:shift}
\end{Rmk}

\begin{Rmk}
Condition \ref{cdt:kappa} can be replaced for the assumption that \( \left\vert \kappa \right\vert \geq \frac{1}{2} \). In the case where \( \kappa \leq -\frac{1}{2} \), the roles of \( u \) and \( v \) as the small and large components of the wave function are swapped.
\end{Rmk}

\subsection{Construction of the regular solution for the generalized radial Dirac system}\label{sec:regular}

Consider the homogeneous radial Dirac system
\begin{equation}\label{eqn:dirachsys}
\begin{cases}
v'+p_1u+(\frac{\kappa}{x}+q )v &= 0, \\
-u'+(\frac{\kappa}{x}+q )u+p_2v &= 0.
\end{cases}
\end{equation}
Assume, for the moment, that both \(p_{1} \) and \( p_2 \) do not vanish anywhere in the interval and that \( p_{1,2}, \, q \in C^1([0,a]) \). Dividing the first equation by \(p_1 \) and the second equation by \(p_2 \) and taking derivatives, two decoupled second-order differential equations can be obtained. After simplification,  the equations are
\begin{align}
u''&-\left( \frac{p_2'}{p_2} \right)u'-\left( \frac{\kappa(\kappa-1)}{x^2}+q_u \right)u=0, \label{eqn:uBhom} \\
v''&-\left( \frac{p_1'}{p_1} \right)v'-\left( \frac{\kappa(\kappa+1)}{x^2}+q_v \right)v=0, \label{eqn:vBhom}
\end{align}
where
\begin{equation}\label{eqn:qu}
q_u=\frac{\kappa(2p_2q-p_2')}{p_2}\frac{1}{x} + p_2 \left( \frac{q}{p_2} \right)'+q^2-p_2p_1,
\end{equation}
and
\begin{equation}\label{eqn:qv}
q_v=\frac{\kappa(2p_1q+p_1')}{p_1}\frac{1}{x} - p_1 \left( \frac{q}{p_1} \right)'+q^2-p_2p_1.
\end{equation}
Each of these equations is a perturbed Bessel equation. Now, any solution \( (u_0, v_0)^T \) of the homogeneous system \eqref{eqn:dirachsys} must necessarily satisfy that \( u_0 \) and \( v_0 \) are solutions to eqs. \eqref{eqn:uBhom} and \eqref{eqn:vBhom} respectively. We introduce some results related to the regularity of solutions of perturbed Bessel equations, in order to provide a precise definition of the regular solution of the radial Dirac system. \\

We recall to the reader that given two functions \( f_1(x) \) and \( f_2(x) \), they are called asymptotically equal as \( x \to x_0 \) iff
\[
f_1(x) = f_2(x)[1+o(1)],
\]
as \( x \to x_0 \). In particular, if \( f_2 \) does not vanish in a neighborhood of the point \( x_0 \), this is equivalent to
\[
\lim_{x\to x_0} \frac{f_1(x)}{f_2(x)} = 1.
\]
When any of these conditions occurs, we write
\[
f_1(x) \sim f_2(x) \quad x \to x_0.
\]
Using the SPPS approach, Castillo, Kravchenko and Torba \cite[Theorem 2.4]{Castillo} showed that for a perturbed Bessel equation of the form
\begin{equation}\label{eqn:bessel}
L \varphi=-\varphi'' + \left( \frac{l(l+1)}{x^2} + q(x) \right) \varphi= \lambda (r_1(x) \varphi' + r_0(x) \varphi),
\end{equation}
with \(  l \geq -\frac{1}{2}, \,\lambda \in \mathbb{C}, \, x \in (0,a], r_i \in C[0,a] \) and \( q(x) = O(x^\alpha), \, x\to 0, \alpha > -2  \); there exists a solution \(\varphi(x;\lambda) \) of \( L \varphi = \lambda \varphi \), that is bounded at the origin \( x=0 \) and satisfies the following asymptotic relations
\begin{align}
\varphi(x) &\sim x^{l+1}, \quad x \to 0, \label{eqn:asymp}\\
\varphi'(x) &\sim (l+1)x^l, \quad x \to 0. \label{eqn:asympd}
\end{align}
The solution satisfying the asymptotic relations above is called a regular solution of the perturbed Bessel equation.\\

For these reasons, throughout this work, by a regular solution  of the radial Dirac system \eqref{eqn:dirac} we mean a solution \( Y=(u,v) \) such that \( u \) and \( v \) are bounded at the origin and satisfy the asymptotic relations
\begin{align}
u(x) &\sim x^\kappa, x \to 0, \quad &u'(x) &\sim \kappa x^{\kappa-1}, x\to 0, \label{eqn:asymu} \\
v(x)&\sim \mu x^{\kappa+1}, x \to 0, \quad  &v'(x) &\sim \mu (\kappa+1) x^\kappa, x \to 0, \label{eqn:asymv}
\end{align}
where \( \mu:=\mu(\lambda) \) is a suitable constant for \( (u,v)^T \) to be a solution of eq. \eqref{eqn:dirac}. If for some particular value of \(\lambda \) one has that \( \mu=0 \), the asymptotic relation \eqref{eqn:asymv} should be read as \( v=o(x^{\kappa+1}) \). \\

We can compute the value of \( \mu \) by direct substitution of the asymptotic relations in the equation
\[
v' + p_1u + \left( \frac{\kappa}{x}+q \right)v=\lambda \left( r_{11}u+r_{12}v \right).
\]
Taking the limit \(x \to 0 \), we get
\begin{equation}
\mu =  \frac{\lambda r_{11}(0)-p_1(0)}{2\kappa+1}.
\end{equation}

Intuitively, we seek the regular solution of the radial Dirac system as a series of the form
\begin{equation}\label{eqn:ideasol}
\begin{pmatrix}
u \\ v
\end{pmatrix} = \sum_{n=0}^{\infty} \lambda^n \begin{pmatrix} u_n \\ v_n \end{pmatrix}.
\end{equation}
Formally substituting the right-hand side of the series in the system \eqref{eqn:dirac}, after comparing the coefficients of the equal powers of the parameter \( \lambda \), we see that the functions \(u_n, \, v_n \) must satisfy the following recursive system of differential equations
\begin{equation}\label{eqn:diracrec}
\begin{cases}
v_n'+p_1u_n+(\frac{\kappa}{x}+q )v_n= r_{11} u_{n-1}+ r_{12} v_{n-1}, \\
-u_n'+(\frac{\kappa}{x}+q )u_n+p_2v_n= r_{21} u_{n-1} + r_{22} v_{n-1},
\end{cases}
\end{equation}
for \( n \geq 0 \), where we take \( u_{-1}, \, v_{-1} \equiv 0 \) to simplify notations. Notice that \( (u_0,v_0 )^T \) is a solution of the homogeneous system
\begin{equation}\label{eqn:dirach}
B \frac{dY}{dx}+ P(x)Y=0.
\end{equation}
In order to have that the series \eqref{eqn:ideasol} satisfies the asymptotic relations \eqref{eqn:asymu} and \eqref{eqn:asymv} it is sufficient that we take  \( u_0 \equiv f, \, v_0 \equiv g  \), where \( (f,g)^T \) is a solution of the homogeneous system \eqref{eqn:dirach} which satisfies \eqref{eqn:asymfg}, and \( u_n, \, v_n \) satisfy the asymptotic relations \eqref{eqn:asymuv1} and \eqref{eqn:asymuvn} for \( n \geq 1 \).
\begin{align}
f(x)&\sim  x^\kappa, \, x \to 0, \quad &g(x)&\sim \frac{-p_1(0)}{2 \kappa+1}x^{\kappa+1}, \, x\to 0,  \label{eqn:asymfg} \\
 u_1(x) &= O(xf(x)), \quad &v_1(x) &\sim \frac{r_{11}(0)}{2 \kappa+1}x^{\kappa+1}, \, x\to 0, \label{eqn:asymuv1} \\
 u_n(x) &=  O(xf(x)), \quad &v_n(x) &= O(xg(x)),  \, n \geq 2. \label{eqn:asymuvn}
\end{align}
With this in mind, our aim is to find solutions to the non-homogeneous systems \eqref{eqn:diracrec} for \(n \geq 1 \), and show that their solutions satisfying \eqref{eqn:asymuv1} and \eqref{eqn:asymuvn}, respectively, are unique.
\begin{Lemma}\label{lemma:nhdirac}
Assume that the homogeneous system \eqref{eqn:dirach} possesses a solution \(Y_0 = (f,g)^T \) such that \( f,\, g \) satisfy the asymptotic relation \eqref{eqn:asymfg} and \( f \) does not vanish except at \( x=0. \) Then, there exists a unique solution to the non-homogeneous system
\begin{equation}\label{eqn:nhdirac}
\begin{cases}
v'+p_1u+(\frac{\kappa}{x}+q )v= h_1, \\
-u'+(\frac{\kappa}{x}+q )u+p_2v= h_2,
\end{cases}
\end{equation}
where \( h_{1,2} \in C([0,a]) \) and \( h_{1,2} = O(x^{\kappa+1}) \), that satisfies the condition \( u(x)=O(xf(x)) \) and \( v(x)=O(xg(x)) \) as \( x\to 0 \). Such solution can be written in the form
\begin{align}
u(x)&=f(x)\int_{0}^{x}\left[ -\frac{h_2(t)}{f(t)}+\frac{p_2(t)}{f^2(t)}\int_{0}^{t}\left[h_1(s)f(s)+h_2(s)g(s)\right]ds \right]dt, \label{eqn:uif} \\
v(x) &= \frac{1}{f(x)} \left[ g(x)u(x)+\int_{0}^{x}\left[h_1(t)f(t)+h_2(t)g(t)\right]dt   \right]. \label{eqn:vif}
\end{align}
\end{Lemma}

\begin{proof}
The proof is similar to that of \cite[Lemma 2.3]{Nelson}, however, we must have special care of the singularities at \(x=0 \). First, we obtain the formulas \eqref{eqn:uif} and \eqref{eqn:vif} in the case when \( h_{1,2},p_{1,2},q \in C^1[0,a] \), the functions \( p_1 \) and \(p_2 \) are non-vanishing on [0,a], both \(f, \,g \) do not vanish except at \( x=0 \) and \( h_{1,2}'(x) = O(x^\kappa)  \). \\

Similarly to the homogeneous case, by differentiation and algebraic transformations, we obtain a pair of decoupled perturbed Bessel equations for \( u \) and \(v \)
\begin{align}
u''&-\left( \frac{p_2'}{p_2} \right)u'-\left( \frac{\kappa(\kappa-1)}{x^2}+q_u \right)u=h_u, \label{eqn:u} \\
v''&-\left( \frac{p_1'}{p_1} \right)v'-\left( \frac{\kappa(\kappa+1)}{x^2}+q_v \right)v=h_v, \label{eqn:v}
\end{align}
where \( q_u, \, q_v \) are given by \eqref{eqn:qu} and \eqref{eqn:qv}, and \( h_u, h_v \) are given by
\begin{align}
h_u&=-p_2\left(\frac{h_2}{p_2} \right)'-\left( \frac{\kappa}{x}+q \right)h_2+p_2h_1, \label{eqn:hu}\\
h_v&=-p_1\left(\frac{h_1}{p_1} \right)'-\left( \frac{\kappa}{x}+q \right)h_1+p_1h_2. \label{eqn:hv}
\end{align}
The pair of functions \( f, \, g \) satisfy the homogeneous part of the equations, i.e, they are solutions of eqs. \eqref{eqn:uBhom} and \eqref{eqn:vBhom}. On the other hand, consider a differential operator of the form
\[
L=\frac{d^2}{dx^2}-\left( \frac{p'}{p} \right) \frac{d}{dx}-\tilde{q}.
\]
where \( p \in C^1[0,a] \) and \( p  \) does not vanish. If \( y_0 \) is a particular solution of the equation \( Ly_0=0 \) such that \( y_0 \) does not vanish except perhaps at \(x=0 \), then it admits the following Polya's factorization (see \cite{Polya}):
\[
L=\frac{p}{y_0} \partial \frac{y_0^2}{p} \partial \frac{1}{y_0}.
\]
The operator in its factored form can be easily inverted. Formally, we can write the solution to \( Ly=h \) using antiderivatives as follows
\[
y(x) = y_0(x) \int_{}^{} \frac{p(t)}{y_0^2(t)} \int_{}^{} \frac{h(s)y_0(s)}{p(s)} ds dt + C_1 y_0(x) \int_{}^{} \frac{p(t)}{y_0^2(t)}dt+C_2 y_0(x).
\]
In particular, any solution \( (u,v)^T \) of system \eqref{eqn:nhdirac} satisfies that
\begin{align}
u(x)&=f(x)\int_{}^{} \frac{p_2(t)}{f^2(t)} \int_{}^{} \frac{h_u(s) f(s)}{p_2(s)}dsdt+ c_{11} f(x) \int_{}^{} \frac{p_2(t)}{f^2(t)}dt+c_{12}f(x), \label{eqn:uns} \\
v(x)&=g(x)\int_{}^{} \frac{p_1(t)}{g^2(t)} \int_{}^{} \frac{h_v(s) g(s)}{p_1(s)}dsdt+ c_{21} g(x) \int_{}^{} \frac{p_1(t)}{g^2(t)}dt+c_{22}g(x), \label{eqn:vns}
\end{align}
for suitable constants \( c_{ij},  \, i,j \in \{1,2\}\). We are left to find appropriate limits of integration in order to have well-defined integrals. \\

First, neither of the expressions
\begin{align*}
\int_{0}^{x} \frac{p_2(t)}{f^2(t)}\int_{0}^{t} \frac{h_u(s) f(s)}{p_2(s)}ds dt&= \int_{0}^{x} \left\{\frac{p_2(t)}{f^2(t)}\int_{0}^{t} \left[ -\left(\frac{h_2(s)}{p_2(s)} \right)'-\left( \frac{\kappa}{s}+q(s) \right)\frac{h_2(s)}{p_2(s)}+h_1(s) \right] f(s) ds\right\}dt, \\
\int_{0}^{x} \frac{p_1(t)}{g^2(t)}\int_{0}^{t} \frac{h_v(s) g(s)}{p_1(s)}ds dt&= \int_{0}^{x} \left\{\frac{p_1(t)}{g^2(t)}\int_{0}^{t} \left[ \left(\frac{h_1(s)}{p_1(s)} \right)'-\left( \frac{\kappa}{s}+q(s) \right)\frac{h_1(s)}{p_1(s)}+h_2(s) \right] g(s) ds\right\} dt,
\end{align*}
is singular at \( x=0 \) due to the fact that \( h_{1,2} = O(x^{\kappa+1}), \, h_{1,2}' = O(x^\kappa)  \). Indeed, both terms in square-brackets are \(O(x^\kappa) \). Therefore, using the asymptotic relations \eqref{eqn:asymfg} for \(f \) and \(g \) we see that
\begin{align*}
\int_{0}^{x} \left[ -\left(\frac{h_2(s)}{p_2(s)} \right)'-\left( \frac{\kappa}{s}+q(s) \right)\frac{h_2(s)}{p_2(s)}+h_1(s) \right] f(s) ds &= O(x^{2\kappa+1}),\\
\int_{0}^{x} \left[ \left(\frac{h_1(s)}{p_1(s)} \right)'-\left( \frac{\kappa}{s}+q(s) \right)\frac{h_1(s)}{p_1(s)}+h_2(s) \right] g(s) ds &= O(x^{2\kappa+2}).
\end{align*}
The asymptotic relations \eqref{eqn:asymfg} also ensure that
\[
 \frac{1}{f(x)}  = O(x^{-\kappa}), \qquad  \frac{1}{g(x)}  = O(x^{-\kappa-1}) \qquad x \in (0,a].
\]
Thus,
\begin{align*}
\int_{0}^{x} \left\{\frac{p_2(t)}{f^2(t)}\int_{0}^{t} \left[ -\left(\frac{h_2(s)}{p_2(s)} \right)'-\left( \frac{\kappa}{s}+q(s) \right)\frac{h_2(s)}{p_2(s)}+h_1(s) \right] f(s) ds\right\}dt = O(x^2),\\
\int_{0}^{x} \left\{\frac{p_1(t)}{g^2(t)}\int_{0}^{t} \left[ \left(\frac{h_1(s)}{p_1(s)} \right)'-\left( \frac{\kappa}{s}+q(s) \right)\frac{h_1(s)}{p_1(s)}+h_2(s) \right] g(s) ds\right\} dt = O(x).
\end{align*}
The terms to the right of \( c_{11} \) and \( c_{21} \) constitute linearly independent solutions to \( f,\,g \) of the homogeneous equations \eqref{eqn:uBhom} and \eqref{eqn:vBhom} respectively. To ensure the existence of the integrals appearing in these expressions for any \( x>0 \), we take the limits of integration from \(a \) to \(x \), i.e.,
\[
f(x) \int_{a}^{x} \frac{p_2(t)}{f^2(t)}dt, \quad
g(x) \int_{a}^{x} \frac{p_1(t)}{g^2(t)}dt.
\]
Any of these integrals with other limits of integration would only differ from our choice by a constant, and therefore, a change of election of the limits of integration would also be an expression the form \eqref{eqn:uns} and \eqref{eqn:vns} with an appropriate adjustment of the constants \(c_{ij} \). \\

Notice that as \(x \) approaches zero, due to the asymptotic relations  \eqref{eqn:asymfg} and condition \ref{cdt:p10} the term \( g(x) \int_{a}^{x}\frac{p_1(t)}{g^2(t)} dt \) is not bounded. \\

Indeed, the condition \ref{cdt:p10} implies that either \( \mathrm{Re}(p_1(0)) \neq 0 \) or \( \mathrm{Im}(p_1(0)) \neq 0 \). Assume for simplicity that \( \mathrm{Re}(p_1(0)) \neq 0 \). Then, there exists \( 0< \delta < a \) such that \( \mathrm{Re}(p_1(0)) \) does not vanish in the interval \( [0,\delta] \) and it does not change sign. Therefore, for \( x \in (0,\delta) \) we have
\[
\int_{a}^{x} \frac{p_1(t)}{g^2(t)}dt = \int_{a}^{\delta} \frac{p_1(t)}{g^2(t)}dt + \int_{\delta}^{x}\frac{p_1(t)}{g^2(t)}dt.
\]
Moreover, there exists \( 0< \beta \leq 1 \) such that
\[
\left\vert  \int_{\delta}^{x}\frac{p_1(t)}{g^2(t)}dt \right\vert \geq \beta \left\vert \mathrm{Re} \int_{\delta}^{x}\frac{p_1(t)}{g^2(t)}dt \right\vert \longrightarrow \infty, \quad x \to 0.
\]
Using L'Hopital's rule, we see that \( g(x) \int_{a}^{x}\frac{p_1(t)}{g^2(t)} dt \sim \frac{2\kappa+1}{\kappa+1}x^{-\kappa}, \, x\to 0. \)
So as to have a bounded solution the coefficient \( c_{21} \) must be zero. \\

On the other hand, since we are looking for a solution that satisfies the asymptotics \( u(x) = O(xf(x)) \) and \( v(x) = O(xg(x)) \) as \(x \to 0 \), the coefficients \( c_{12}  \) and \( c_{22} \) must be zero as well. Since \( p_2(x) \) is a continuous arbitrary function, we cannot claim the same for \( c_{11}  \) at the moment. \\

Using integration by parts and recalling that \( f(0)=0 \) and \( g(0)=0 \), we obtain the formulas
\begin{align}
u(x)&=f(x)\int_{0}^{x}\left[ -\frac{h_2(t)}{f(t)}+\frac{p_2(t)}{f^2(t)}\int_{0}^{t}h_1(s)f(s)+h_2(s)g(s)ds \right]dt + c_{11}f(x)\int_{a}^{x} \frac{p_2(t)}{f^2(t)}dt, \label{eqn:utmp} \\
v(x)&=g(x)\int_{0}^{x}\left[ \frac{h_1(t)}{g(t)}+\frac{p_1(t)}{g^2(t)}\int_{0}^{t}h_1(s)f(s)+h_2(s)g(s)ds \right]dt. \end{align}
Setting \( w(t):=1/(f(t)g(t)) \) and \( z(t):=\int_{0}^{t}\left(f(s)h_1(s)+g(s)h_2(s) \right)ds \), we have that for any \( x>0 \)
\begin{align*}
\int_{0}^{x} \left[\frac{h_1(t)}{g(t)}+ \frac{p_1(t)}{g^2(t)}z(t)\right]dt &= \int_{0}^{x} \left[ w(t) \left( h_1(t)f(t)+h_2(t)g(t) \right) - \frac{h_2(t)}{f(t)}+ \frac{p_2(t)}{f^2(t)}z(t)\right.\\
&\quad\quad\quad\left.+ \left( \frac{p_1(t)}{g^2(t)}-\frac{p_2(t)}{f^2(t)} \right)z(t)\right] dt \\
&= \frac{u(x)}{f(x)}+ \int_{0}^{x} \left[ w(t) \left( h_1(t)f(t)+h_2(t)g(t) \right) + \left( \frac{p_1(t)}{g^2(t)}-\frac{p_2(t)}{f^2(t)} \right)z(t)\right] dt.
\end{align*}
Moreover, it can checked that \( (fg)'=p_2g^2-p_1f^2 \). Consequently \( \left( \frac{1}{fg} \right)'=\frac{p_1}{g^2}-\frac{p_2}{f^2} \). Thus,
\begin{equation}
\int_{0}^{x} \left[\frac{h_1(t)}{g(t)}+ \frac{p_1(t)}{g^2(t)}z(t)\right]dt=\frac{u(x)}{f(x)}+ \int_{0}^{x} \left[w(t)z'(t)+w'(t)z(t)\right]dt = \frac{u(x)}{f(x)}+\left(w(t)z(t) \right)\Big\vert_{0}^x.
\end{equation}
Since \( h_1(x)=O(x^{\kappa+1}) \), \( \lim_{x \to 0} w(x)z(x)=0 \), we have that
\begin{align}
v(x)&=g(x)\left[\frac{u(x)}{f(x)}+ \frac{1}{f(x)g(x)}\int_{0}^{x}\left[h_1(t)f(t)+h_2(t)g(t)\right]dt   \right] \label{eqn:vif-aux}\\
&=\frac{1}{f(x)} \left[ g(x)u(x)+\int_{0}^{x}\left[h_1(t)f(t)+h_2(t)g(t)\right]dt   \right], \label{eqn:vtmp}
\end{align}
which shows \eqref{eqn:vif}. Notice that if we show that \( (u,v)^T \) given by eqs. \eqref{eqn:uif} and \eqref{eqn:vif} is a solution of eq. \eqref{eqn:nhdirac}, by linearity, substitution of formulas \eqref{eqn:utmp} and \eqref{eqn:vtmp} into the system \eqref{eqn:nhdirac} gives us
\begin{equation*}
\begin{cases}
v'+p_1u+(\frac{\kappa}{x}+q )v= h_1+c_{11}\widehat{f}, \\
-u'+(\frac{\kappa}{x}+q )u+p_2v= h_2+c_{11}\left(-\widehat{f}\,'+(\frac{\kappa}{x}+q)\widehat{f}\right),
\end{cases}
\end{equation*}
where \( \widehat{f}(x):=f(x)\int_{a}^{x}\frac{p_2(t)}{f^2(t)}dt, \) which implies that \( c_{11} \) must equal zero. \\

We therefore proceed to show that this is indeed the case, without any additional assumptions on \( h_{1,2}, p_{1,2} \) and \( g \). First, we write
\begin{equation}
v= \frac{z}{f} + \frac{u}{f}g.
\end{equation}
Then,
\begin{equation}
v'=\frac{z'f-f'z}{f^2}+ \left( \frac{u}{f} \right)'g + \left( \frac{u}{f} \right)g'.
\end{equation}
On the other hand, using that \( (f,g)^T \) is a solution of the homogeneous system we get
\begin{equation}
\frac{z'f-f'z}{f^2}= \frac{1}{f^2} \left[ (h_1f+h_2 g)f - \left[ \left( \frac{\kappa}{x}+q \right)f +p_2g \right] z \right].
\end{equation}
Also,
\begin{equation}
g \left( \frac{u}{f} \right)' =g \left[  -\frac{h_2}{f} + \frac{p_2}{f^2}z\right],
\end{equation}
and
\begin{equation}
g'\left( \frac{u}{f} \right) = -p_1 u - \left( \frac{\kappa}{x}+q \right) g \frac{u}{f}.
\end{equation}
Consequently,
\begin{equation}
v'+p_1 u + \left( \frac{ \kappa}{x} +q \right)v= h_1.
\end{equation}
Finally, we can see that
\begin{equation}
-u'= -f' \frac{u}{f} + h_2 - \frac{p_2}{f}z = \left[ -\left( \frac{\kappa}{x}+q \right)f - p_2g\right]\frac{u}{f}+h_2-\frac{p_2}{f}z,
\end{equation}
which implies that
\begin{equation}
-u'+\left( \frac{\kappa}{x}+q \right)u+p_2v=h_2.
\end{equation}
Therefore, system \eqref{eqn:nhdirac} possesses a unique solution \( (u,v )^T \) given by eqs. \eqref{eqn:uif} and \eqref{eqn:vif} such that \(u \) and \(v \) satisfy the asymptotics \( u(x)=O(xf(x)) \) and \( v(x)=O(xg(x)) \) respectively.
\end{proof}

The next result verifies the validity of formulas \eqref{eqn:uif} and \eqref{eqn:vif} for a particular case where the hypothesis \( h_{1,2} = O(x^{\kappa+1}) \) is not satisfied, which corresponds to the first case \( n=1 \) of the recursive system \eqref{eqn:diracrec}.
\begin{Lemma}\label{lemma:rfgdirac}
Under the conditions of Lemma \ref{lemma:nhdirac},
there is a unique solution to the system
\begin{equation}\label{eqn:diracrfg}
\begin{cases}
v'+p_1u+(\frac{\kappa}{x}+q )v= r_{11} f + r_{12} g, \\
-u'+(\frac{\kappa}{x}+q )u+p_2v= r_{21} f + r_{22} g,
\end{cases}
\end{equation}
that satisfies the asymptotics \( u(x)=O(xf(x)) \) and \( v(x) \sim \frac{r_{11}(0)}{2\kappa+1}x^{\kappa+1} \) as \( x\to 0 \), given by
\begin{align}
u(x)&=f(x)\int_{0}^{x}\left[ -r_{21}(t)-r_{22}(t)\frac{g(t)}{f(t)}+\frac{p_2(t)}{f^2(t)}z(t) \right]dt, \label{eqn:u1} \\
v(x)&= \frac{1}{f(x)}z(x)+g(x)\frac{u(x)}{f(x)}  \label{eqn:v1},
\end{align}
where
\[
z(x):=\int_{0}^{x}\left[ f^2(t)r_{11}(t)+f(t)g(t)r_{12}(t)+f(t)g(t)r_{21}(t)+g^2(t)r_{22}(t) \right]dt.
\]
\end{Lemma}

\begin{proof}
For this proof, we set \( h_1(x) := r_{11}(x)f(x)+r_{12}(x)g(x) \) and \( h_2(x):=r_{21}(x)f(x)+r_{22}(x)g(x)\). Using the same reasoning as in the proof of Lemma \ref{lemma:nhdirac}, first we assume that \( p_{1,2}, r_{ij} \in C^1([0,a]) \) and \( p_{1,2} \) do not vanish anywhere in the interval. Also, we first assume that \( g \) does not vanish except at \( x=0 \). The procedure is almost the same. At first glance, it is not clear that the term
\begin{equation}\label{eqn:tmp}
\int_{0}^{x} \left[ \frac{p_1(t)}{g^2(t)}\int_{0}^{t} \left[ \left(\frac{h_1(s)}{p_1(s)} \right)'-\left( \frac{\kappa}{s}+q(s) \right)\frac{h_1(s)}{p_1(s)}+h_2(s) \right] g(s) ds \right] dt,
\end{equation}
is well-defined, because of division by \(g^2 \) which is \( O(x^{2\kappa+2}) \) and the previous estimate for \( h_1 \) cannot be used. However, using that \( (f,g)^T \) is a solution to the homogeneous Dirac system, we get that
\begin{align*}
\left(\frac{h_1(x)}{p_1(x)} \right)'-\left( \frac{\kappa}{x}+q(x) \right)\frac{h_1(x)}{p_1(x)}+h_2(x)&=f'(x)\frac{r_{11}(x)}{p_1(x)}+\left(\frac{r_{11}(x)}{p_1(x)} \right)'f(x)+\left( \frac{r_{12}(x)g(x)}{p_1(x)} \right)'\\
&\,-\left( \frac{\kappa}{x}+q(x) \right)\frac{r_{11}(x)f(x)}{p_1(x)} -\left( \frac{\kappa}{x}+q(x) \right)\frac{r_{12}(x)g(x)}{p_1(x)}+h_2(x)\\
&= \frac{g(x)r_{11}(x)p_2(x)}{p_1(x)}+\left(\frac{r_{11}(x)}{p_1(x)} \right)'f(x)+\left( \frac{r_{12}(x)g(x)}{p_1(x)} \right)'\\
&\, -\left( \frac{\kappa}{x}+q(x) \right)\frac{r_{12}(x)g(x)}{p_1(x)}+h_2(x)\\
&=O(x^\kappa),
\end{align*}
which shows that \eqref{eqn:tmp} is well-defined. Using integration by parts, it is possible to obtain similar formulas to those given by \eqref{eqn:utmp} and \eqref{eqn:vtmp}, which are
\begin{align}
u(x)&=f(x)\int_{0}^{x}\left[ -\frac{h_2(t)}{f(t)}+\frac{p_2(t)}{f^2(t)}\int_{0}^{t}h_1(s)f(s)+h_2(s)g(s)ds \right]dt + \tilde{c}_{11}f(x)\int_{a}^{x} \frac{p_2(t)}{f^2(t)}dt, \label{eqn:utmp2} \\
v(x)&=g(x)\int_{0}^{x}\left[ \frac{h_1(t)}{g(t)}+\frac{p_1(t)}{g^2(t)}\int_{0}^{t}h_1(s)f(s)+h_2(s)g(s)ds \right]dt+\tilde{c}_{22}g(x). \label{eqn:vtmp2}
\end{align}
We have already dropped the terms corresponding to the constants \( \tilde{c}_{12}, \tilde{c}_{21} \) with an analogous argument.
The simplification of the integral expression in \eqref{eqn:vtmp2} can be achieved just as in the previous lemma. We obtain
\begin{equation}
\int_{0}^{x} \left[\frac{h_1(t)}{g(t)}+ \frac{p_1(t)}{g^2(t)}z(t)\right]dt=\frac{u(x)}{f(x)}+ \int_{0}^{x} \left[w(t)z'(t)+w'(t)z(t)\right] dt = \frac{u(x)}{f(x)}+\left(w(t)z(t) \right)\Big\vert_{0}^x.
\end{equation}
where \( w(x)=\frac{1}{f(x)g(x)} \). Contrary to the previous lemma, using L'Hopital's rule and the fact that \( (fg)'=p_2g^2-p_1f^2 \) now we have
\[
\lim_{x\to 0} w(x)z(x)= \lim_{x\to 0} \frac{\int_{0}^{x}[h_1(t)f(t)+h_2(t)g(t)]dt}{f(x)g(x)} = \lim_{x\to 0} \frac{r_{11}(x)f^2(x)}{p_2(x)g^2(x)-p_1(x)f^2(x)} = -\frac{r_{11}(0)}{p_1(0)}.
\]
In order to have a solution that satisfies \( v(x) \sim \frac{r_{11}(0)}{2\kappa+1}x^{\kappa+1}, \, x \to 0 \), we take \( c_{22}= \frac{r_{11}(0)}{p_1(0)} \) to get
\begin{align*}
u(x)&=f(x)\int_{0}^{x}\left[ -r_{21}(t)-r_{22}(t)\frac{g(t)}{f(t)}+\frac{p_2(t)}{f^2(t)}z(t) \right]dt+\tilde{c}_{11} f(x) \int_{a}^{x} \frac{p_2(t)}{f^2(t)}dt,  \\
v(x)&=g(x) \left[ \frac{1}{f(x)g(x)}z(x)+\frac{u(x)}{f(x)} \right]=\frac{1}{f(x)}z(x)+g(x)\frac{u(x)}{f(x)}.
\end{align*}
Finally, we can verify without any additional hypothesis on \( p_{1,2}, \, r_{i,j} \) and \( g \) that \( (u,v)^T \) given by the formula above cannot be a solution to eq. \eqref{eqn:diracrfg} unless  \( \tilde{c}_{11}=0. \) Therefore, \( (u,v)^T \) where \(u \) and \(v \) are given by eqs. \eqref{eqn:u1} and \eqref{eqn:v1}, is indeed the unique solution to eq. \eqref{eqn:diracrfg} that satisfies the asymptotics \( u(x)=O(xf(x)) \) and \( v(x) \sim \frac{r_{11}(0)}{2\kappa+1}x^{\kappa+1} \) as \( x\to 0 \).

\end{proof}
Now, we introduce the system of recursive functions known as formal powers. As in other works of the SPPS method, the idea is to rewrite formulas \eqref{eqn:uif}, \eqref{eqn:vif}, \eqref{eqn:u1} and \eqref{eqn:v1} in such a way that only one integration is performed on each step. To keep consistency with the notation in \cite{Nelson} and maintain some symmetry between the formulas, we first consider the following system of functions
\begin{align}
X^{(0)}(x) &= 1, \\
Y^{(0)}(x) &= 1, \\
Z^{(n)}(x) &= \int_{0}^{x} \left( X^{(n)}(t)\left( f^2(t)r_{11}(t) +f(t)g(t)r_{21}(t)\right)+
            Y^{(n)}(t)\left( f(t)g(t)r_{12}(t) +g^2(t)r_{22}(t) \right) \right) dt, \label{eqn:Zn} \\
X^{(n+1)}(x) &= (n+1)\int_{0}^{x} \left( -r_{21}(t)X^{(n)}(t)-r_{22}(t)\frac{g(t)}{f(t)}Y^{(n)}(t)+\frac{p_2(t)}{f^2(t)}Z^{(n)}(t) \right) dt,\label{eqn:Xn} \\
Y^{(n+1)}(x) &= (n+1)\left( \frac{1}{f(x)g(x)}Z^{(n)}(x) \right)+X^{(n+1)}(x), \quad n\geq 0. \label{eqn:Yn}
\end{align}
Then we have that
\begin{equation}
\begin{pmatrix} u_n \\ v_n \end{pmatrix} = \frac{1}{n!} \begin{pmatrix} fX^{(n)} \\ gY^{(n)} \end{pmatrix}, \quad n \geq 0.
\end{equation}
\begin{Rmk}
Formula \eqref{eqn:Zn} differs from formula (2.6) in \cite{Nelson} due to a minor mistake in that work.
\end{Rmk}
The right-hand side of \eqref{eqn:Yn} is based on the intermediate-step formula \eqref{eqn:vif-aux} and presents the drawback of requiring \( g \) not to vanish except at \( x=0. \) As it is only necessary that \( f \) does not vanish in both Lemma \ref{lemma:nhdirac} and Lemma \ref{lemma:rfgdirac}, we depart slightly from the usual way of presenting the formal powers and introduce the following formulas instead:
\begin{align}
\widehat{X}^{(0)}(x)&=f(x)\\
\widehat{Y}^{(0)}(x)&=g(x)\\
\widehat{Z}^{(n)}(x)&= \int_{0}^{x} \left( \widehat{X}^{(n)}(t)\left( f(t)r_{11}(t) +g(t)r_{21}(t)\right)+
            \widehat{Y}^{(n)}(t)\left( f(t)r_{12}(t) +g(t)r_{22}(t) \right) \right) dt, \label{eqn:Znhat} \\
\widehat{X}^{(n+1)}(x) &= (n+1)f(x)\int_{0}^{x} \left( -\frac{r_{21}(t)}{f(t)}\widehat{X}^{(n)}(t)-\frac{r_{22}(t)}{f(t)}\widehat{Y}^{(n)}(t)+\frac{p_2(t)}{f^2(t)}\widehat{Z}^{(n)}(t) \right) dt, \label{eqn:Xnhat} \\
\widehat{Y}^{(n+1)}(x) &= (n+1)\left( \frac{1}{f(x)}\widehat{Z}^{(n)}(x) \right)+\frac{g(x)}{f(x)}\widehat{X}^{(n+1)}(x), \quad n\geq 0. \label{eqn:Ynhat}
\end{align}
The solution of the Dirac system \eqref{eqn:dirac} can be then written as a series of the form
\[
\begin{pmatrix} u \\ v \end{pmatrix} = \sum_{n=0}^{\infty} \frac{\lambda^n}{ n!} \begin{pmatrix} \widehat{X}^{(n)} \\ \widehat{Y}^{(n)} \end{pmatrix}.
\]
Using the asymptotic relations  \eqref{eqn:asymfg} and the fact that \( f \) is non-vanishing except at \( x=0 \), we ensure the existence of constants \( \tilde{c_1}, \, \tilde{c_2},  \) such that for any \( x \in (0,a] \) we have
\begin{align}
\vert f(x) \vert &\leq \tilde{c_1}x^\kappa,  & \vert g(x) \vert \leq \tilde{c_1}x^{\kappa+1} \\
 \max \{ \left\vert \frac{1}{f(x)} \right\vert,   \frac{\vert p_2(x) \vert^{1/2}}{\vert f(x) \vert}  \} &\leq \tilde{c_2}x^{-\kappa}. &
\end{align}
Additionally, we define the constants \( c_1,c_2,c_3  \) and \( \tilde{a} \) as
\begin{align}
c_1&:= \max \{ \tilde{c_1}, 1 \}, \label{eqn:c1}\\
c_2&:= \max \{ \tilde{c_2}, 1 \},\label{eqn:c2}\\
c_3&:= \max_{1\leq i,j \leq 2} \{ \Vert r_{ij} \Vert_\infty \}, \label{eqn:c3} \\
\tilde{a}&:= \max \{ 1, a \}. \label{eqn:atilde}
\end{align}
\begin{Lemma}\label{lemma:bounds}
Suppose that the homogeneous system \eqref{eqn:dirachsys} possesses a solution \(Y_0 = (f,g)^T \) such that \( f\) does not vanish except at \( x=0 \) and \(f, \, g \) satisfy the asymptotic relations \eqref{eqn:asymfg}. Additionally, assume that the constants \( c_1, c_2, c_3 \) satisfy \eqref{eqn:c1} to \eqref{eqn:c3}, and \( \tilde{a} \) satisfies \eqref{eqn:atilde}. Then the following estimates hold for the functions \( \widehat{X}^{(n)}, \, \widehat{Y}^{(n)}  \) for \( n \geq 1 \)

\begin{align}
\vert \widehat{X}^{(n)}(x) \vert & \leq c_1  M^n x^{n+\kappa}, \label{eqn:boundX}\\
\vert \widehat{Y}^{(n)}(x) \vert & \leq c_1  M^n x^{n+\kappa}, \label{eqn:boundY}
\end{align}
where
\begin{align*}
M&:=c_2M_Z+c_1c_2\tilde{a}M_X, \\
M_Z&:= c_1c_3(1+\tilde{a})^2, \\
M_X&:= c_1c_2c_3[(1+\tilde{a})+c_1c_2(1+\tilde{a})^2 a].
\end{align*}
\end{Lemma}
\begin{proof}
The proof is by induction. First, we verify directly that the formula holds for \( n=1. \) Indeed,
\begin{align*}
\left\vert \widehat{Z}^{(0)}(x) \right\vert &\leq \int_{0}^{x}\left[c_1 t^\kappa (c_1c_3t^\kappa+c_1c_3t^{\kappa+1})+c_1 t^{\kappa+1} (c_1c_3t^\kappa+c_1c_3t^{\kappa+1})  \right]dt \\
&\leq  c_1^2c_3 \int_{0}^{x} \left[ t^{2\kappa}(1+a)+at^{2\kappa}(1+a) \right] dt  \\
&= c_1^2c_3(1+a)^2 \frac{x^{2\kappa+1}}{2\kappa+1}.
\end{align*}
Then,
\begin{align*}
\left\vert \widehat{X}^{(1)}(x) \right\vert &\leq c_1 x^\kappa \int_{0}^{x} \left[ c_1c_2c_3+c_1c_2c_3a+c_1^2c_2^2c_3(1+a)^2a \right]dt \\
&\leq c_1 \left( c_1c_2c_3[(1+a)+c_1c_2(1+a)^2a] \right) x^{\kappa+1}.
\end{align*}
Summarizing, it holds that
\begin{align*}
\left\vert \widehat{Z}^{(0)}(x) \right\vert &\leq c_1 M_Z \frac{x^{2\kappa+1}}{2\kappa+1}, \\
\left\vert \widehat{X}^{(1)}(x) \right\vert &\leq c_1 M_X x^{\kappa+1}.
\end{align*}
Finally,
\begin{align*}
\left\vert \widehat{Y}^{(1)}(x) \right\vert &\leq c_1\left[c_2M_Z+c_1c_2\tilde{a}M_X \right]x^{\kappa+1}=c_1Mx^{\kappa+1}.
\end{align*}
Noting that \( M>M_X \), the result holds for \( n=1 \). Now, assume that \( n\geq 1 \) and that \eqref{eqn:boundX} and \eqref{eqn:boundY} are valid. Then
\begin{align*}
\left\vert \widehat{Z}^{(n)}(x) \right\vert & \leq c_1 M^n \int_{0}^{x} t^{2\kappa+n}(c_1c_2)(1+\tilde{a})+t^{2\kappa+n}(c_1c_2)(1+\tilde{a})\tilde{a} dt \\
&=c_1 M^n M_Z \frac{x^{2\kappa+n+1}}{2\kappa+n+1}.
\end{align*}
Using the bound for \(  \widehat{Z}^{(n)}(x) \) we also see that
\begin{align*}
\left\vert \widehat{X}^{(n+1)}(x) \right\vert & \leq (n+1)c_1x^k\int_{0}^{x}\left[ c_2c_3t^{-\kappa}(c_1M^nt^{n+\kappa})+c_2c_3t^{-\kappa}(c_1M^nt^{n+\kappa})\tilde{a}+c_2^2t^{-2\kappa}(c_1M^nM_Zt^{2\kappa_+n}a) \right] dt \\
&=c_1M^nM_Xx^{n+\kappa+1}.
\end{align*}
Therefore,
\begin{align*}
\left\vert \widehat{Y}^{(n+1)}(x) \right\vert & \leq (n+1)c_2x^{-\kappa}\left( c_1M^nM_Z \frac{x^{2\kappa+n+1}}{2\kappa+n+1} \right)+c_1c_2\tilde{a}x\left( c_1M^nM_X x^{n+\kappa+1} \right)\\
&\leq c_1 M^{n+1}x^{n+\kappa+1}
\end{align*}
Since \( M>M_X \), it follows that also \(\left\vert \widehat{X}^{(n)}(x) \right\vert \leq c_1 M^n x^{n+\kappa} \) holds true.
\end{proof}

The following theorem presents the spectral parameter power series (SPPS) representation of a regular solution of eq. \eqref{eqn:dirac}.

\begin{Thm}\label{thm:spps}
Suppose that the homogeneous system \eqref{eqn:dirachsys} possesses a solution \(Y_0 = (f,g)^T \) such that \( f\) and \( g \) satisfy the asymptotic relations \eqref{eqn:asymfg} and \( f \) does not vanish except at \( x=0 \). Then, a regular solution of the system \eqref{eqn:dirac} is given by
\begin{equation}\label{eqn:spps}
\begin{pmatrix}
u \\
v
\end{pmatrix} = \sum_{n=0}^{\infty} \frac{\lambda^n}{n!} \begin{pmatrix}
\widehat{X}^{(n)} \\
\widehat{Y}^{(n)}
\end{pmatrix}
\end{equation}
\end{Thm}
\begin{proof}
By Lemma \ref{lemma:bounds} the series converges uniformly on any compact subset \( K \subset (0,a] \) as well as the series of termwise derivatives. Hence, we can apply the Dirac operator \( B\frac{d}{dx}+P \) termwise to this series. Similarly to the proof of \cite[Theorem 2.2]{Nelson}, using eqs. \eqref{eqn:Xnhat} and \eqref{eqn:Ynhat} and Lemmas \ref{lemma:rfgdirac} and \ref{lemma:nhdirac} we argue that the functions \( \widehat{X}^{(n)}, \widehat{Y}^{(n)}, n\geq 0 \) satisfy
\[
B \frac{d}{dx} \begin{pmatrix} \widehat{X}^{(n)} \\ \widehat{Y}^{(n)} \end{pmatrix}
+P(x) \begin{pmatrix} \widehat{X}^{(n)} \\ \widehat{Y}^{(n)} \end{pmatrix}
= n \cdot R(x) \begin{pmatrix} \widehat{X}^{(n-1)} \\ \widehat{Y}^{(n-1)} \end{pmatrix}
\]
Consequently,
\[\left( B \frac{d}{dx} +P \right) \begin{pmatrix}
u \\
v
\end{pmatrix} = \sum_{n=0}^{\infty} \frac{\lambda^n}{n!} \cdot n R  \begin{pmatrix} \widehat{X}^{(n-1)} \\ \widehat{Y}^{(n-1)} \end{pmatrix}
 = \lambda R \sum_{n=1}^{\infty} \frac{\lambda^{n-1}}{(n-1)!} \begin{pmatrix} \widehat{X}^{(n-1)} \\ \widehat{Y}^{(n-1)} \end{pmatrix}
= \lambda R \begin{pmatrix}
u \\
v
\end{pmatrix}.
\]
In the series above \( \widehat{X}^{(-1)}, \, \widehat{Y}^{(-1)} \equiv 0 \).
We are just left to justify why such solution is regular.
Again, by Lemma \ref{lemma:bounds} we see that
\begin{align*}
u(x)&=x^{\kappa}+o(x^{\kappa}), \\
v(x)&= \mu x^{\kappa+1} + o(x^{\kappa+1}),
\end{align*}
Consequently, the series of formal powers satisfy asymptotics \eqref{eqn:asymu} and \eqref{eqn:asymv}.
\end{proof}

\subsection{Spectral shift technique}\label{sec:shift}

The SPPS representation of the regular solution of the Dirac equation is based on a particular solution of eq. \eqref{eqn:dirac} for \( \lambda =0 \). For the classical Sturm-Liouville equation, the spectral shift technique was introduced in \cite{SPPS}. Since then, the spectral shift technique has been successfully applied to improve the numerical accuracy in numerous publications of Sturm-Liouville problems, see e.g. \cite{Castillo, Castillo2015, KRAVCHENKO201482, Nelson} and the references therein. \\

Recently, the spectral shift technique for the Dirac equation was introduced in \cite{Nelson}. We recall the procedure here for readability. First, for a symmetric matrix \( R \), assume there exists a particular solution \( (u,v)^T \) of eq. \eqref{eqn:dirac} for \( \lambda = \lambda_0 \) such that \( u \) does not vanish except at \( x=0 \). If \(p_1(0)-\lambda_0 r_{11}(0) \neq 0, \) we can subtract \( \lambda_0 R(x) Y(x) \) from both sides of eq. \eqref{eqn:dirac} to rewrite the system as
\begin{equation}\label{eqn:diracshift}
B \frac{dY}{dx} + \left( P(x) - \lambda_0R(x) \right)Y = (\lambda-\lambda_0) R(x) Y.
\end{equation}
Since we assume that \( R \) is a symmetric matrix, eq. \eqref{eqn:diracshift} is of the same type as \eqref{eqn:dirac} and \( (u,v)^T \) is a particular solution of eq. \eqref{eqn:diracshift} with non-vanishing first entry that corresponds to \( \lambda-\lambda_0 =0 \). Therefore, we can construct the spectral parameter power series with respect to the spectral parameter \( \Lambda:=\lambda-\lambda_0 \) using Theorem \ref{thm:spps}. \\

In the case where \( R \) is not symmetric and \(p_1(0)-\lambda_0 r_{11}(0) \neq 0, \) the change of variables
\[
Y = w(x)U, \quad w(x) = \exp \left( -\frac{\lambda_0}{2} \int_{}^{} \mathrm{tr} BR(s) ds \right),
\]
transforms system \eqref{eqn:dirac} to the system
\begin{equation}\label{eqn:diracshiftns}
B \frac{dU}{dx} + \left( P(x) - \lambda_0R(x) - \frac{\lambda_0}{2}B \, \mathrm{tr} BR(x) \right)U = (\lambda-\lambda_0) R(x) U,
\end{equation}
which is a system of type \eqref{eqn:dirac}. Therefore, if \( (u,v)^T \) is a particular solution for eq. \eqref{eqn:dirac} such that \(u \) does not vanish except at \( x=0 \), \( \frac{1}{w(x)}(u,v)^T \) is a particular solution for eq. \eqref{eqn:diracshiftns} whose first entry is non-vanishing except at \( x=0 \) and we can construct the SPPS representation with respect to \( \Lambda= \lambda - \lambda_0 \). \\

Contrary to other publications of the SPPS method, where the spectral shift technique can be applied unrestrictively, in our work there is one particular case for the parameter \(\lambda_0 \) where it cannot be applied, that is, when \(p_1(0)-\lambda_0 r_{11}(0) = 0. \) However, this is not a great inconvenience as one can take another suitable spectral shift by \( \tilde{\lambda_0}\) that is near to \( \lambda_0 \).

\subsection{Construction of a particular solution}\label{sec:particular}

In this subsection we show how to construct a particular solution of eq. \eqref{eqn:dirachsys} satisfying asymptotics \eqref{eqn:asymfg}. To do so, we rewrite the system as
\begin{equation}\label{eqn:dirachsysrw}
\begin{cases}
v'+p_1u+\frac{\kappa}{x} v= -qv \\
-u'+\frac{\kappa}{x}u = -p_2v -qu,
\end{cases}
\end{equation}
The homogeneous part of system \eqref{eqn:dirachsysrw} possesses a solution
\begin{equation}\label{eqn:sln1}
(f_0,g_0)^T = (x^\kappa, -x^{-\kappa}\int_{0}^{x}t^{2\kappa} p_1(t) dt)^T,
\end{equation}
that satisfies the asymptotics \eqref{eqn:asymfg} due to condition \ref{cdt:p10}. The system \eqref{eqn:dirachsysrw} is a particular case of \eqref{eqn:dirac} with \( \lambda =1 \). We note that \( f_0 \) does not vanish except at \( x=0 \) and Theorem \ref{thm:spps} can be applied. Moreover, since in this case we have \( r_{11} \equiv 0 \), the parameter \(\mu \) is given by \( \mu=-\frac{p_1(0)}{2\kappa+1} \) and the resulting series satisfies the asymptotics \eqref{eqn:asymfg} as desired.

\subsection{The case of a singular potential}\label{sec:singq}

We now show that it is possible to use the method to construct a solution of  system \eqref{eqn:dirachsysrw} that satisfies the asymptotics \eqref{eqn:asymfg} for a potential \( q\in C(0,a] \) with a growth bound at the origin
\begin{equation}\label{eqn:cq}
\left\vert q(x) \right\vert \leq c_q x^\alpha,
\end{equation}
for some \( c_q>0 \) and some \( -1<\alpha\leq0 \).  \\

Consider the solution \(Y_0 = (f_0,g_0)^T \) of the homogeneous part of \eqref{eqn:dirachsysrw} given by \eqref{eqn:sln1}, and consider the following recursive system of functions
\begin{align}
\widehat{X}^{(0)}_q(x)&=f_0(x)\\
\widehat{Y}^{(0)}_q(x)&=g_0(x)\\
\widehat{Z}^{(n)}_q(x)&= \int_{0}^{x} \left(- \widehat{X}^{(n)}_q(t)g_0(t)q(t)+
            \widehat{Y}^{(n)}_q(t)\left(-f_0(t)q(t)-g_0(t)p_2(t) \right) \right) dt, \label{eqn:Znq} \\
\widehat{X}^{(n+1)}_q(x) &= (n+1)f_0(x)\int_{0}^{x} \left( \frac{q(t)}{f_0(t)}\widehat{X}^{(n)}_q(t)+\frac{p_2(t)}{f_0(t)}\widehat{Y}^{(n)}_q(t)\right) dt, \label{eqn:Xnq} \\
\widehat{Y}^{(n+1)}_q(x) &= (n+1)\left( \frac{1}{f_0(x)}\widehat{Z}^{(n)}_q(x) \right)+\frac{g_0(x)}{f_0(x)}\widehat{X}^{(n+1)}_q(x), \quad n\geq 0. \label{Ynq}
\end{align}

Since \( p_2 \in C[0,a] \) there exists a constant \( c_p \) such that
\begin{equation}\label{eqn:cp}
\max_{x \in [0,a]} \left\vert p_2(x) \right\vert \leq c_p.
\end{equation}

\begin{Lemma}\label{lemma:boundsq}
Consider the solution \(Y_0 = (f_0,g_0)^T \) of the homogeneous part of \eqref{eqn:dirachsysrw} given by \eqref{eqn:sln1} for a potential with a growth condition at the origin given by \eqref{eqn:cq}. Also, suppose that the constants \( c_1, c_2, c_p \) satisfy \eqref{eqn:c1}, \eqref{eqn:c2}  and \eqref{eqn:cp}. Then the following estimates hold for the functions \( \widehat{X}^{(n)}_q, \, \widehat{Y}^{(n)}_q \) for \( n \geq 0 \)
\begin{align}
\vert \widehat{X}^{(n+1)}_q(x) \vert & \leq c_1\frac{M_q^n}{(1+\alpha)^n} x^{n(1+\alpha)+\kappa}, \label{eqn:boundXq}\\
\vert \widehat{Y}^{(n+1)}_q(x) \vert & \leq c_1\frac{M_q^n}{(1+\alpha)^n} x^{n(1+\alpha)+\kappa+1}, \label{eqn:boundYq}
\end{align}
where
\begin{align*}
M_q&:=c_2M_{Z_q}+c_1c_2M_{X_q}, \\
M_{Z_q}&:=c_1(2c_q+c_p\tilde{a}^{1-\alpha}),\\
M_{X_q}&:=c_1c_2\left( c_q+c_p \tilde{a}^{1-\alpha} \right).
\end{align*}
\end{Lemma}

\begin{proof}
The proof is similar to that of Lemma \ref{lemma:bounds}.
\end{proof}
The following result follows directly from the bounds of Lemma \ref{lemma:boundsq}.
\begin{Prop}\label{prop:fqgq}
Consider the solution \(Y_0 = (f_0,g_0)^T \) of the homogeneous part of \eqref{eqn:dirachsysrw} given by \eqref{eqn:sln1} for a potential with a growth condition at the origin given by \eqref{eqn:cq}. Then, the solution of \eqref{eqn:dirachsys} that satisfies the asymptotics \eqref{eqn:asymfg} is given by
\begin{equation}\label{eqn:fqgq}
\begin{pmatrix}
f_q \\
g_q
\end{pmatrix} = \sum_{n=0}^{\infty} \frac{1}{n!} \begin{pmatrix}
\widehat{X}^{(n)}_q \\
\widehat{Y}^{(n)}_q
\end{pmatrix}
\end{equation}
\end{Prop}
We can now generalize Theorem \ref{thm:spps} to obtain the SPPS representation for a system with a singular potential.
\begin{Thm}
Consider the Dirac system \eqref{eqn:dirac} with a potential that satisfies the growth condition at the origin \eqref{eqn:cq}. Let \(Y_{0,q} = (f_q,g_q)^T \) be the solution of the homogeneous system \eqref{eqn:dirachsys} given by Proposition \ref{prop:fqgq}. Assume that \( f_q \) does not vanish except at \( x=0 \).  Then, the regular solution of the system \eqref{eqn:dirac} is given by \eqref{eqn:spps}.
\end{Thm}
\begin{proof}
We note that the proof of Theorem \ref{thm:spps} remains valid if the homogeneous system \eqref{eqn:dirachsys} possesses a solution \( Y_0= (f, g)^T \) such that \( f \) and \( g \) satisfy the asymptotic relations \eqref{eqn:asymfg} and \( f \) does not vanish except at \( x=0 \). According to Proposition \ref{prop:fqgq}, the solution \( Y_{0,q} \) satisfies the required asymptotics, and together with the assumption for \( f_q \), taking \( Y_0=Y_{0,q} \) the result follows.
\end{proof}

\subsection{Solutions with a non-vanishing first entry}\label{sec:nonvanishing}
The existence of a non-vanishing particular solution has been established in the case of regular Sturm-Liouville problems (see \cite[Remark 5]{SPPS}). Moreover, in the case of the one-dimensional Dirac equation, the existence and construction of a non-vanishing particular solution has been proven as well using an appropriate linear combination of the two linearly independent solutions, see \cite[Proposition 2.9, Proposition 2.10]{Nelson}. However, for the radial Dirac equation there is only one solution satisfying the asymptotics \eqref{eqn:asymu} and \eqref{eqn:asymv}. \\

It is known that for real-valued coefficients \(p_{1,2} \) and real-valued potential \( q \), the Dirac operator \( B \frac{d}{dx}+P(x) \) has real spectrum. If we have that \( R \equiv I \) where \( I \) denotes the \(2\times2 \) identity matrix, a spectral shift by \( \lambda_0 \) with \( \mathrm{Im}(\lambda_0) \neq 0 \) must necessarily have a particular solution \( (\tilde{u},\tilde{v})^T \) such that \( \tilde{u} \) does not vanish except at \( x=0 \). \\

Indeed, suppose that \( (\tilde{u},\tilde{v} )^T \) is a solution of \(B\frac{d}{dx}Y+P(x)Y=\lambda_0Y \) in the interval \( (0,a] \) and there exists a point \( x_0 \in (0,a) \) such that \( \tilde{u}(x_0)=0 \). Then, we consider the following spectral problem
\begin{equation}
\begin{cases}
v'+p_1u+(\frac{\kappa}{x}+q )v=\lambda u, \qquad x\in (0,x_0] \\
-u'+(\frac{\kappa}{x}+q )u+p_2v=\lambda v, \\
u(x_0) = 0.
\end{cases}
\end{equation}
Then \( (\tilde{u},\tilde{v})^T \) is an eigenfunction of the above spectral problem corresponding to the complex eigenvalue \( \lambda=\lambda_0 \), which is impossible. \\

On the other hand, notice that due to the asymptotic relations \eqref{eqn:asymu} and \eqref{eqn:asymv} that the regular solution \((u,v)^T \) satisfies, there exists \(  0<\delta<a \) such that both \( u \) and \( v \) do not vanish in the interval \( (0,\delta] \).
Then,  we can work with the equation
\begin{equation}
\begin{cases}
v'+p_1u+(\frac{\kappa}{x}+q )v=\lambda u, \qquad x\in [\delta,a] \\
-u'+(\frac{\kappa}{x}+q )u+p_2v=\lambda v, \\
\end{cases}
\end{equation}
which is no longer singular and compute the solution of the system in the interval \( [\delta,a] \) that continuously extends the regular solution from the interval \( (0,\delta ] \) to the whole interval \( (0,a] \). For more details about the SPPS method for the Dirac equation in the one-dimensional case, see \cite{Nelson}.

\section{Numerical implementation and examples}\label{sec:numerical}

We consider the spectral problem given by the following boundary conditions
\begin{equation}\label{eqn:bc}
(\alpha_1,\alpha_2)Y(a) = 0,
\end{equation}
where \( \alpha_1, \, \alpha_2 \) are complex constants satisfying  \( \left\vert \alpha_1 \right\vert+ \left\vert \alpha_2 \right\vert \neq 0 \). The application of the SPPS representation to solve the spectral problem for the system \eqref{eqn:dirac} is similar to that of \cite{Nelson, Castillo}. The algorithm is as follows.

\begin{enumerate}
\item Find a particular solution \( (f,g )^T \) of the homogeneous system \eqref{eqn:dirachsys} satisfying the asymptotic conditions \eqref{eqn:asymfg}. If an analytic expression for the particular solution is unknown, one can use a numerical approximation with a truncated series of eq. \eqref{eqn:fqgq} described in Section \ref{sec:particular}. Notice that \eqref{eqn:fqgq} covers both the singular and continuous potential.
\item Check that \(f \) is non-vanishing for \( x\in(0,a] \). If the particular solution has zeros inside the interval \( (0,a] \), the spectral shift technique may be used finding a suitable value of \( \lambda_0 \), complex in general.
\item The solution \eqref{eqn:spps} satisfies the boundary condition if and only if the following characteristic equation is satisfied:
\[
\Delta(\lambda):= \alpha_1 u(a; \lambda) + \alpha_2 v(a;\lambda)=0.
\]
Use partial sums of the series \eqref{eqn:spps} to obtain a polynomial
\begin{equation}\label{eqn:chareq-spps}
\Delta_N(\lambda):= \alpha_1 \sum_{n=0}^{N} \frac{\lambda^n}{n!} \widehat{X}^{(n)}(a) + \alpha_2 \sum_{n=0}^{N} \frac{\lambda^n}{n!}\widehat{Y}^{(n)}(a),
\end{equation}
that approximates the characteristic function \( \Delta(\lambda) \).
\item Find the zeros of the polynomial \( \Delta_N(\lambda) \). Its roots closest to the origin approximate the exact eigenvalues. To discard spurious eigenvalues, see the discussion in \cite{Castillo}.
\item If larger eigenvalues are needed, apply several steps of the spectral shift technique. The following procedure yielded excellent results.\\

In the first step, choose a complex number \(\eta_0 = \sigma+i\tau_0  \) with \( 0< \left\vert \tau_0 \right\vert< \left\vert \sigma \right\vert \). Take a finite set of real constants \( A:=\{\beta_j \}_{j=1}^k  \), such that \( 0<\beta_1<\ldots<\beta_l=1<\beta_{l+1}<\ldots<\beta_k \). Compute the solution of
\[
B\frac{dY}{dx}+PY=\lambda Y,
\]
for each value of \( \lambda=\lambda_{1,j}\) with \( \lambda_{1,j}=\sigma+i\beta_j\tau_0, \, j=1,\ldots,k. \) Select the solution having the least value of the expression
\begin{equation}\label{eqn:min}
\max_{j=1,\ldots,k} \left\{ \left\Vert u_j(x) \right\Vert, \left\Vert v_j \right\Vert, \left\Vert \frac{x^k}{u_j(x)} \right\Vert \right\},
\end{equation}
where \( \Vert u \Vert:=\max_{x\in[0,a]} \left\vert u(x) \right\vert \) and \( (u_j,v_j)^T \) is the solution corresponding to \( \lambda_{1,j}. \) The first shift is given by \( \lambda_1:=\sigma+\gamma_1 \tau_0 \), where \( \gamma_1\in A \) is the value that minimizes \eqref{eqn:min}. \\

For the subsequent steps proceed similarly. The \( n \)-th shift is given by \( \lambda_n := n\sigma+i\tau_n^* \) where \( \tau_n^* \) is the number that has the least absolute value between \( \sigma \) and
\[
\tau_n=\gamma_n \tau_{n-1}=\gamma_n \gamma_{n-1} \ldots \gamma_1 \tau_0,
\]
and \( \gamma_n \) is chosen as the value of \( \beta_j \) that minimizes \eqref{eqn:min} for the solutions corresponding to \( \lambda_{n,j}=n\sigma +i \beta_j \tau_{n-1} \). \\

\end{enumerate}
In the numerical examples we used a uniform mesh with \(M \) points to represent all functions involved, and a Newton-Cotes 6-point rule modified to perform indefinite integration. All the computations where done in Matlab 2020b with double machine precision.

\subsection{Application to spectral problems for the perturbed Bessel equation}\label{sec:bessel}

Consider a perturbed Bessel equation of the form
\begin{equation}\label{eqn:bess}
-u''+\left(\frac{l(l+1 )}{x^2}+q_B(x)\right)u = \omega^2 r(x) u
\end{equation}
where \( q_B, \, r \in C(0,a) \) are complex-valued functions such that \( r(0)\neq 0 \), and \( l \geq -1/2 \). Suppose that \( u_0 \) is a regular solution that does not vanish except at \( x=0 \) corresponding to \( \omega=0 \). Let \( \omega \neq 0 \). We introduce the variable \( v \) by the equation
\begin{equation}
\omega v = -u' +\frac{u_0'}{u_0}u.
\end{equation}
Then, eq. \eqref{eqn:bess} is equivalent to the system,
\begin{equation}
\begin{cases}
v' + \frac{u_0'}{u_0}u &= \omega r u, \\
-u' + \frac{u_0'}{u_0}v &= \omega v.
\end{cases}
\end{equation}
We rewrite this system as
\begin{equation}\label{eqn:b2daux}
\begin{cases}
v' +ru+ \frac{u_0'}{u_0}v &= (\omega+1) r u, \\
-u' + \frac{u_0'}{u_0}u+v &= (\omega+1) v,
\end{cases}
\end{equation}
so that it satisfies the condition \ref{cdt:p10}. Moreover, for a potential \( q_B \in C(0,a) \) that satisfies the growth condition \(  q_B(x) = O( x^\beta) \) with \(  -2 < \beta \leq 0 \), using \cite[eqs. 3.5-3.7, 3.9]{Castillo} we see that
\begin{equation}
\frac{u_0'}{u_0} = \frac{l+1}{x}+O(x^{\beta+1}).
\end{equation}
We define \( q_D(x):=\frac{u_0'}{u_0}-\frac{\kappa_l}{x} \), with \( \kappa_l = l+1 \). We note that the potential \( q_D \) satisfies the growth condition at the origin \eqref{eqn:cq} for \( \alpha=\beta+1>-1 \) and we can work with this case via the results in Section \ref{sec:particular}. Then, the system \eqref{eqn:b2daux} can be written as
\begin{equation}
\begin{cases}
v' +ru+ \left( \frac{\kappa_l}{x}+q_D(x)\right)v &= (\omega+1) r u, \\
-u' + \left( \frac{\kappa_l}{x}+q_D(x)\right)u+v &= (\omega+1) v,
\end{cases}
\end{equation}
The remaining question is if we can compute a particular solution of the homogeneous part of \eqref{eqn:bess} using the Dirac system approach, or more precisely, if we can compute the function \( q_D \). To answer this question, we write
\[
v_0 = u_0'-\frac{\kappa_l}{x}u_0.
\]
After differentiation we arrive to the system
\begin{equation}
\begin{cases}
v_0' +u_0+ \frac{\kappa_l}{x}v_0 &= (q_B(x)+1)u_0\\
-u_0' + \frac{\kappa_l}{x}u_0 &=  -v_0,
\end{cases}
\end{equation}
and we take \( q_D(x):=\frac{v_0(x)}{u_0(x)} \). \\

Following \cite[Sec. 4.1]{Nelson}, we transform a spectral problem for eq. \eqref{eqn:bess} with a boundary condition of the form
\begin{equation}
\alpha_1 u(a)+ \alpha_2 u'(a)=0,
\end{equation}
to the equivalent boundary condition (which depends on the spectral parameter)
\begin{equation}
\left( \alpha_1 + \frac{\alpha_2 u_0'(a)}{u_0(a)} \right) u(a) - \omega \alpha_2 v(a) =0.
\end{equation}

\begin{Ex}\label{Ex:73}
Consider the following spectral problem
\begin{equation}
\begin{cases}
-y''+\left( \frac{\nu^2-\frac{1}{4}}{x^2}+x^2 \right)y = \lambda y, \quad 0<x\leq \pi \\
y(\pi,\lambda)=0,
\end{cases}
\end{equation}
with \( \nu=2. \) For this problem, the exact characteristic equation is known in terms of the Whittaker-M function; see \cite{Nelson}. We can compute the exact eigenvalues with any desired precision using, e.g., Mathematica. In table \ref{tab:ex73}, we present the results obtained by the SPPS method and the exact values for the first ten eigenvalues without using the spectral shift technique. \\

The accuracy of the computed eigenvalues is dependent on both parameters \(N \) and \(M \), where \(N \) is the truncation parameter that appears in eq. \eqref{eqn:chareq-spps} and \(M \) stands for the number of points used in the integration. In figure \ref{fig:err73v1}, we can see that taking a constant step size in the spectral shift technique can be quite unstable, even for small variations of the parameter \( \tau \). Although the accuracy improves with larger N and M, the results are still not satisfactory. \\

On the contrary, the adaptive step size of the spectral shift technique as described in the algorithm provides excellent results. The numerical experiments
suggest that as long as we take a not so small value of \( \tau_0 \) the
strategy provides very similar results. For a fixed parameter of \(N = 50\), increasing the number of integration points from \(M = 50000\) to \(M = 100000\) has a significant impact on the number of eigenvalues that can be computed with good accuracy, see figure \ref{fig:err73v2}. In the case of the constant step size in the spectral shift, such an impact is nowhere near as significant as with the additive step size.

\begin{table}[tb]
\centering
\begin{tabular}{cccc}
\hline
\(n \) & \( \sqrt{\lambda_n}  \text{ (Dirac) } \) &  \( \sqrt{\lambda_n} \text{(Exact)} \) & \( \text{ Abs. error }\) \\
\hline
   1    &      2.46294997397399    &      2.46294997397397  &    \(1.909\cdot 10^{-14}\)\\
   2    &      3.28835292994174    &      3.28835292994256  &    \(8.255\cdot 10^{-13}\)\\
   3    &      4.14986421876492    &      4.14986421874478  &    \(2.014\cdot 10^{-11}\)\\
   4    &      5.06366882344452    &       5.0636688237341  &    \(2.895\cdot 10^{-10}\)\\
   5    &        6.007581461806    &        6.007581458116  &    \(3.690\cdot 10^{-09}\)\\
   6    &      6.96849408294089    &      6.96849412676994  &    \(4.382\cdot 10^{-08}\)\\
   7    &       7.9397377451917    &       7.9397373768993  &    \(3.682\cdot 10^{-07}\)\\
   8    &      8.91769491485358    &       8.9176967111442  &    \(1.796\cdot 10^{-06}\)\\
   9    &      9.90029802733174    &      9.90026271201188  &    \(3.531\cdot 10^{-05}\)\\
  10    &      10.8845351338648    &      10.8861250916173  &    \(1.589\cdot 10^{-03}\)\\
\hline
\end{tabular}
\caption{First ten eigenvalues of Example \ref{Ex:73} withouth any spectral shift.}
\label{tab:ex73}
\end{table}

\begin{figure}[h]
\centering
 \includegraphics[width=6in,height=4in]{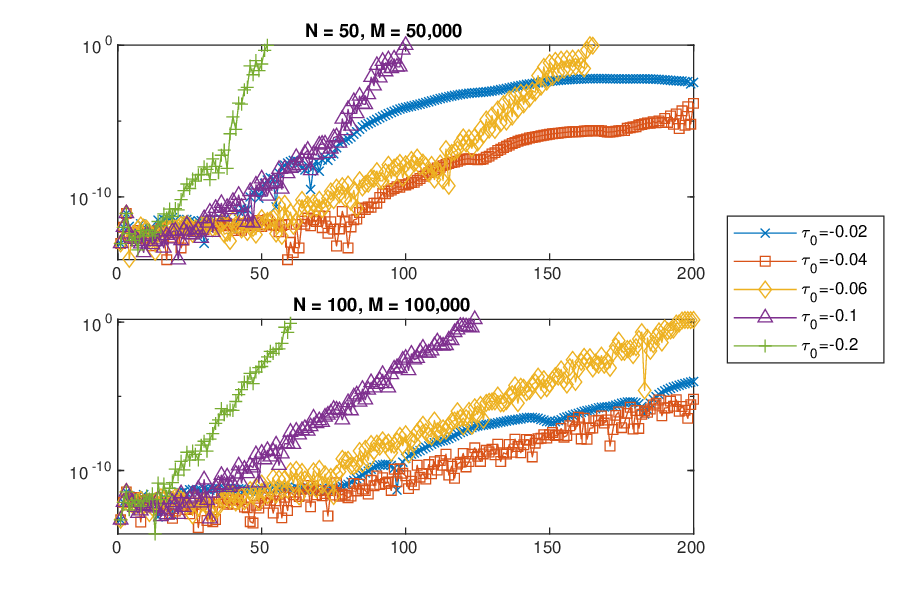}
\caption{Absolute error of \( \sqrt{\lambda_n} \) of Example \ref{Ex:73} using the spectral shift technique with a constant step size \( \lambda_0^{(n)}=n \sigma+ n \tau_0 i   \) for different values of \( \tau_0, \, \sigma=-2 \).}
\label{fig:err73v1}
\end{figure}

\begin{figure}
\centering
 \includegraphics[width=6in,height=4in]{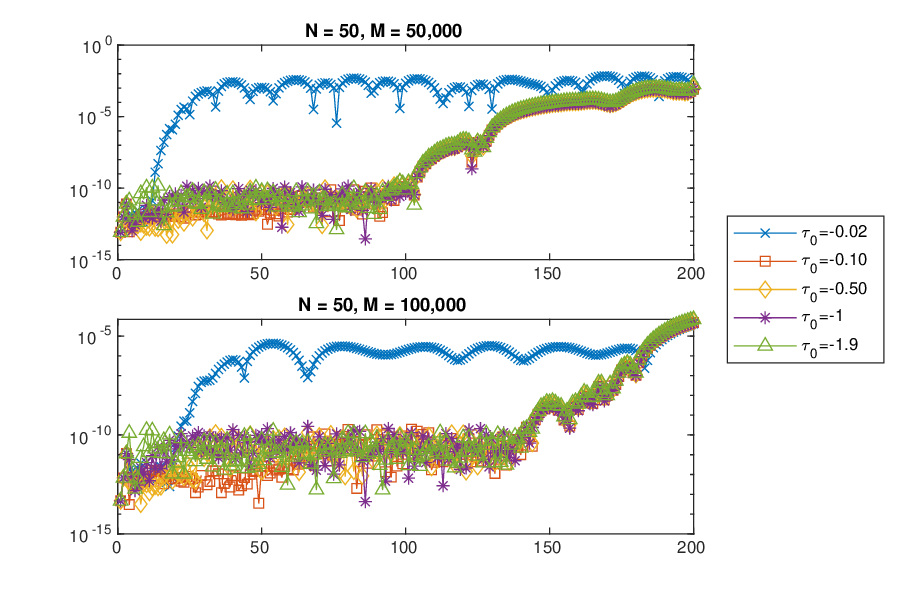}
\caption{Absolute error of \(\sqrt{\lambda_n} \) of Example \ref{Ex:73} using the spectral shift technique with an adaptative step size, with the seed value \( \eta_0=\sigma+\tau_0 i  \) for different values of \( \tau_0, \, \sigma=-2 \). The set of dilation constants is \( A = \{ 0.9, 1, 1.1 \}\).}
\label{fig:err73v2}
\end{figure}

\end{Ex}

\begin{Ex}\label{Ex:boyd}
We consider an application to a singular potential. Consider the Boyd equation together with the following boundary condition
\begin{equation}
\begin{cases}
-y''-\frac{1}{x}y=\lambda y, \quad 0<x\leq 1, \\
y(1,\lambda)=0.
\end{cases}
\end{equation}
In fig. \ref{fig:errboydv2} we see the absolute error of the square root of the eigenvalues \( \sqrt{\lambda_n} \) using the adaptive spectral shift strategy. Similarly to the previous example, not so small values of \( \tau_0 \) relative to \( \sigma \) provide good accuracy allowing one to compute at least a hundred eigenvalues.
\begin{figure}
\centering
 \includegraphics[width=5in,height=2.4in]{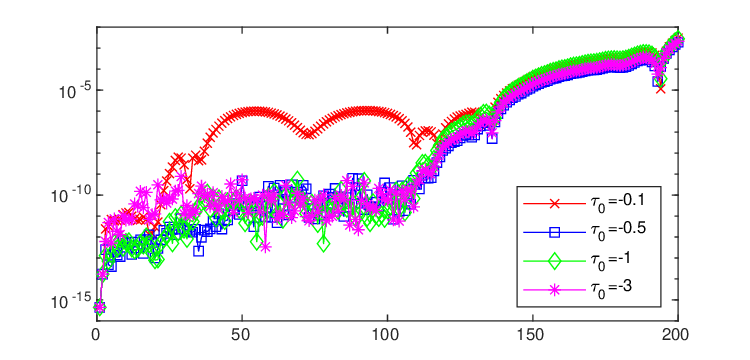}
\caption{Absolute error of \(\sqrt{\lambda_n} \) of Example \ref{Ex:boyd} using the spectral shift technique with an adaptive step size, with the seed value \( \eta_0=\sigma+\tau_0 i  \) for different values of \( \tau_0, \, \sigma=-6 \). The set of dilation constants is \( A = \{ 0.9, 1, 1.1 \}, \, N=100, \, M=100000 \).}
\label{fig:errboydv2}
\end{figure}
\end{Ex}

\begin{Rmk}
It is worth mentioning that our method based on the SPPS representation for the radial Dirac equation outperforms the direct approach based on the SPPS representation for the regular solution of perturbed Bessel equations \cite{Castillo}. It reduces the computational cost of the spectral shift technique and gives us the possibility to compute more eigenvalues with good precision.
\end{Rmk}

\subsection{Application of the method for a hydrogen-like atom with finite radius}\label{sec:atom}

We consider the problem of computing the energy levels of an electron orbiting around an atom with finite-radius \( R \). According to \cite{dong}, the \(d \)-dimensional Dirac equation can be reduced to the system of equations
\begin{align}\label{eqn:dirac-coulomb}
\begin{split}
\frac{d}{dx}G(x) + \frac{\kappa}{x}G(x)&=[\bar{E}-V(x)-M]F(x) \\
-\frac{d}{dx}F(x) +\frac{\kappa}{x}F(x) &= [\bar{E}-V(x)+M]G(x)
\end{split}
\end{align}
where \( \kappa=\pm(2l+d-1)/2 \).
We implicitly consider the case where \(d = 3\). For the numerical examples, we work with atomic units, so \( M=m_e·c \) and \( \bar{E}=E/c \), where \(c \) is the speed of light and  \(m_e \) is the mass of the electron. Moreover, we assume a uniform charge distribution within the atomic radius, and a Coulomb potential for \(x>R \), that is
\begin{equation}
V(x)=\begin{cases}
-\frac{\xi}{2R}(3-\frac{x^2}{R^2}), \quad &0<x\leq R \\
 - \frac{\xi}{x}, \quad &R<x
\end{cases}
\end{equation}
where \(\xi = Z \alpha \), \( Z \) is the nuclear charge number and \( \alpha \sim 1/137\) is the fine structure constant. Observe that the system \eqref{eqn:dirac-coulomb} is of type \eqref{eqn:dirac} in the interval \((0,R] \) and the SPPS method can be applied. We look for a solution in the semi-axis \( (0,\infty) \) such that it satisfies the asymptotics \eqref{eqn:asymu}, \eqref{eqn:asymv} at the origin, and it has a decaying behaviour at infinity.\\

For the purpose of readability we present the confluent hypergeometric equation approach as in \cite[Chapter 13]{dong} to compute the exact solutions of Hydrogen-like atoms with a Coulomb potential with proper modifications. Naturally, since we assume a finite-size radius, the confluent hypergeometric procedure with a Coulomb potential is valid for \( x>R \).
 We are interested in energy bound states \( \left\vert \bar{E} \right\vert<M \). Consider a new variable
\begin{equation}
\rho = 2x \sqrt{M^2-\bar{E}^2}.
\end{equation}
The change of variables leads to the system
\begin{align}\label{eqn:diracrho}
\begin{split}
\frac{d}{d\rho}G(\rho) + \frac{\kappa}{\rho}G(\rho)&=\left( - \frac{1}{2} \sqrt{\frac{M-\bar{E}}{M+\bar{E}}}+\frac{\xi}{\rho} \right) F(\rho), \\
\frac{d}{d\rho}F(\rho) -\frac{\kappa}{\rho}F(\rho) &= \left( - \frac{1}{2} \sqrt{\frac{M+\bar{E}}{M-\bar{E}}}-\frac{\xi}{\rho} \right)G(\rho).
\end{split}
\end{align}
Introduce a new pair of functions \( \Phi_{\pm}(\rho) \) such that
\begin{eqnarray}
G(\rho) &= \sqrt{M-\bar{E}} \left[ \Phi_+(\rho) + \Phi_-(\rho) \right], \\
F(\rho) &= \sqrt{M+\bar{E}} \left[ \Phi_+(\rho) - \Phi_-(\rho) \right].
\end{eqnarray}
Substituting into \eqref{eqn:diracrho} yields
\begin{align}\label{eqn:diracphis}
\begin{split}
\left\{ \frac{d}{d \rho} \Phi_+(\rho) + \frac{d}{d \rho} \Phi_-(\rho) \right\}+ \frac{\kappa}{\rho} \left[\Phi_+(\rho) + \Phi_-(\rho) \right] &= \left\{-\frac{1}{2}+ \frac{\xi}{\rho} \sqrt{\frac{M+\bar{E}}{M-\bar{E}}} \right\}\left[\Phi_+(\rho) - \Phi_-(\rho) \right]\\
\left\{ \frac{d}{d \rho} \Phi_+(\rho) - \frac{d}{d \rho} \Phi_-(\rho) \right\}- \frac{\kappa}{\rho} \left[\Phi_+(\rho) - \Phi_-(\rho) \right] &= \left\{-\frac{1}{2}- \frac{\xi}{\rho} \sqrt{\frac{M-\bar{E}}{M+\bar{E}}} \right\}\left[\Phi_+(\rho) + \Phi_-(\rho) \right]
\end{split}
\end{align}
Denote
\begin{equation}
\tau = \frac{\xi \bar{E}}{\sqrt{M^2-\bar{E}^2}}, \quad \tau'=\frac{\xi M}{\sqrt{M^2-\bar{E}^2}}.
\end{equation}
Addition and subtraction of \eqref{eqn:diracphis} can be simplified to obtain
\begin{equation}\label{eqn:dphipm}
\frac{d}{d \rho} \Phi_\pm(\rho) \mp \left( \frac{\tau}{\rho}-\frac{1}{2} \right) \Phi_\pm(\rho) = - \frac{\kappa \pm \tau'}{\rho} \Phi_\mp(\rho).
\end{equation}
From here, we can obtain a pair of decoupled second-order differential equations
\begin{equation}\label{eqn:dphi}
\left\{ \frac{d^2}{d \rho^2} + \frac{1}{\rho} \frac{d}{d \rho} + \left( -\frac{1}{4}+\frac{\tau \mp 1/2}{\rho}- \frac{\eta^2}{\rho^2} \right) \right\} \Phi_\pm (\rho) = 0,
\end{equation}
where
\begin{equation}
\eta = \sqrt{\kappa^2-\xi^2}>0.
\end{equation}
Define
\begin{equation}\label{eqn:Phipm}
\Phi_{\pm}(\rho) = \rho^\eta e^{-\rho/2} R_\pm(\rho).
\end{equation}
Substitution of this into \eqref{eqn:dphi} gives us
\begin{equation}\label{eqn:Rpm}
\begin{split}
\frac{d^2}{d \rho^2}R_{\pm}(\rho) + \left( -1 + \frac{1+2\eta}{\rho} \right) \frac{d}{d \rho} R_{\pm}(\rho) + \frac{\tau -\eta - \frac{1}{2} \pm \frac{1}{2}}{\rho} R_{\pm}(\rho)=0,
\end{split}
\end{equation}
We look for solutions of \eqref{eqn:Rpm}  so that the functions \( \Phi_\pm \) vanish at infinity. On the other hand, equation \eqref{eqn:Rpm} is a confluent hypergeometric equation. Because of the required behaviour at infinity, the looked-for solutions are confluent hypergeometric functions of the second kind, which exhibit bounded behaviour at infinity, see \cite[Sec. 9.10-9.12, f. 9.12.3]{Lebedev}. Therefore, we have that
\begin{align}\label{eqn:Tripm}
\begin{split}
R_+(\rho) &= c_+ U(\eta-\tau, 1+2\eta; \rho) \\
R_-(\rho) &= c_- U(1+\eta-\tau, 1+2\eta; \rho)
\end{split}
\end{align}
where \( U(\alpha, \gamma;z) \) denotes the confluent hypergeometric function of the second kind with parameters \( \alpha, \gamma \) and \( z \) is a complex variable. \\

The function \( U(\alpha, \gamma;z) \) possesses the following differentiation formula \cite[f. 9.10.12]{Lebedev}
\begin{equation}\label{eqn:diffU}
\frac{d}{dz} U(\alpha, \gamma, z) = -\alpha U(\alpha+1, \gamma+1;z).
\end{equation}
Using \eqref{eqn:diffU}, substitution of \eqref{eqn:Phipm} and \eqref{eqn:Tripm} into \eqref{eqn:dphipm} gives us
\begin{equation}\label{eqn:cpm-aux}
(\kappa+\tau')c_-U(1+\eta-\tau,1+2\eta;\rho)+(\eta-\tau)c_+\left[U(\eta-\tau,1+2\eta;\rho)-\rho U(1+\eta-\tau,2+2\eta;\rho) \right]=0.
\end{equation}
We recall  the recursive identity  for the confluent hypergeometric functions of the second kind \cite[f. 9.12.2]{Lebedev}.
\begin{equation}\label{eqn:confluent}
U(\alpha, \gamma; z) = z U(\alpha+1, \gamma+1; z) + \left( 1 + \alpha - \gamma \right) U(\alpha+1, \gamma; z).
\end{equation}
Using \eqref{eqn:cpm-aux} and \eqref{eqn:confluent} we obtain
\begin{equation}\label{eqn:cpm}
c_- = \frac{\eta^2-\tau^2}{\kappa+\tau'} c_+.
\end{equation}
Now, denote \( (\tilde{F}(x), \tilde{G}(x))^T \) the solution of system \eqref{eqn:dirac-coulomb} in the interval \( (0,R] \) that satisfies asymptotics \eqref{eqn:asymu} and \eqref{eqn:asymv}. Then, in order to have a solution to \eqref{eqn:dirac-coulomb} in the semi-axis that satisfies asymptotics \eqref{eqn:asymu} and \eqref{eqn:asymv} and decays at infinity the solution  \( (\tilde{F}(x), \tilde{G}(x))^T \) and the solution  \( ({F}(x), {G}(x))^T \) must be a multiple of one another at the point \(x =R \). Then, the following characteristic equation must hold
\begin{equation}
\Delta:=\begin{vmatrix}
\tilde{F} & F \\
\tilde{G} & G
\end{vmatrix}(R)=\tilde{F}(R)G(R)-F(R)\tilde{G}(R)=0.
\end{equation}
\begin{Ex}
In this example, we work with hydrogen-like oxygen, that is, \( Z=8. \) The radius of the oxygen atom is taken from \cite[Table I]{atomic-radii}. Note that in the previous reference, atomic radii are expressed in angstroms and must be converted to atomic units of length. For the numerical results, we use the numerical value \(c=137.036 \) for the speed of light, and we take \( \kappa =2 \). We applied the spectral shift technique to improve numerical accuracy  of the results with a shift by \( \bar{E}_0=137 \).
For the ease of readability, we present the results for \( E-m_ec^2=c\bar{E}-m_ec^2 \) in Table \ref{tab:oxygen}. In fig. \ref{fig:oxygen} we present the plot of the small and large wave components.
\begin{table}
\centering
\begin{tabular}{cc}
\(n \) &  \(E-m_ec^2 \) \\
\hline
1   &       -4.9982701713634 \\
2   &      -2.57809776595968 \\
3   &      -1.56812208970587 \\
4   &      -1.05246356970019 \\
5   &     -0.754592132634571 \\
6   &     -0.567255972709972 \\
7   &     -0.441874608535727 \\
8   &     -0.353877007739356 \\
9   &     -0.289760463067068 \\
10  &      -0.241608021409775 \\
\hline
\end{tabular}
\caption{First ten energy values for the hydrogen like oxygen with a finite size radius shifted by \(m_ec^2 \).}
\label{tab:oxygen}
\end{table}
\begin{figure}[t]
\centering
\begin{subfigure}{\textwidth}
 \centering
 \includegraphics[width=5in,height=2.4in]{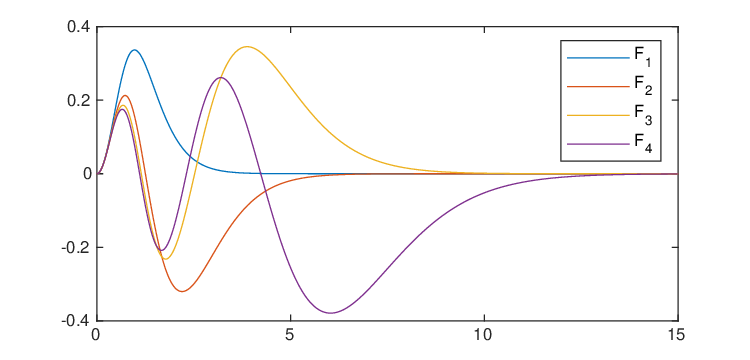} \caption{}
\label{subfig:oxygenG}
\end{subfigure}\\
\begin{subfigure}{\textwidth}
 \centering
 \includegraphics[width=5in,height=2.4in]{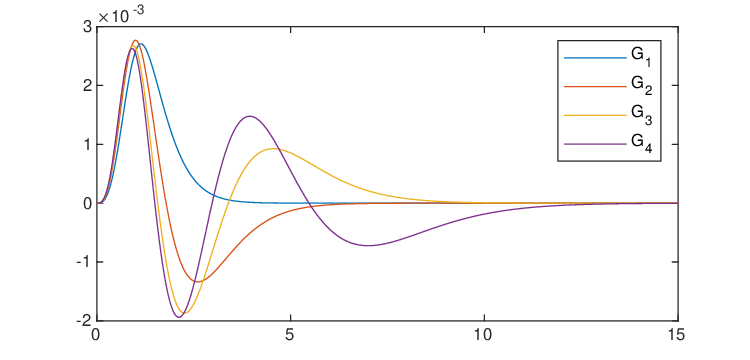} \caption{}
\label{subfig:oxygenF}
\end{subfigure}
\caption{Graph (a) and (b) corresponds to the first four large and small wave components for the hydrogen-like oxygen with finite size radius respectively. }
\label{fig:oxygen}
\end{figure}
\end{Ex}

\end{document}